\documentclass[12pt,a4paper]{article}
\usepackage{subfigure}
\usepackage{amssymb}
\usepackage{amsfonts}
\usepackage{amsmath,amsthm}

\theoremstyle{plain}
\newtheorem{theorem}{Theorem}[section]
\newtheorem*{Peano}{Theorem}{\bf}{\it}

\newtheorem{corollary}[theorem]{Corollary}
\newtheorem{lemma}[theorem]{Lemma}
\newtheorem{proposition}[theorem]{Proposition}
\newtheorem{example}[theorem]{Example}
\theoremstyle{remark}
\newtheorem{remark}{Remark}[section]
\theoremstyle{definition}
\newtheorem{definition}{Definition}[section]
\newcommand{\R}{\mathbb R}
\newcommand{\RR}{\overline{\mathbb R}}
\newcommand{\E}{\mathbb E}
\newcommand{\ld}[3]{#2^{(#1)}_{-} #3}
\newcommand{\lsubd}[2]{\mathrm{\partial}^{(#1)}_{-} #2}
\newcommand{\subd}[2]{{\rm \partial} #1 #2}

\newcommand{\gin}[3]{#2^{[#1]}_{-} #3}
\newcommand{\dini}[2]{#2^{(#1)}_D}
\newcommand{\pr}{\prime}
\newcommand{\eps}{\varepsilon}
\newcommand{\f}{f : \E\to\R\cup\{+\infty\}}

\newcommand{\norm}[1]{\Vert#1\Vert}
\newcommand{\dom}[1]{{\rm dom}\,#1}

\newcommand{\be}{\begin{equation}}
\newcommand{\ee}{\end{equation}}
\title{Higher-order optimality conditions with an arbitrary non-differentiable function}
\author{Vsevolod I. Ivanov\thanks{Email: vsevolodivanov@yahoo.com
\vspace{6pt} }
\\\vspace{6pt}{\em{\small Department of Mathematics, Technical University of Varna, 9010 Varna, Bulgaria} }
}

\begin{document}
\maketitle
\begin{abstract}
In this paper, we introduce a new higher-order directional derivative and higher-order subdifferential of Hadamard type of a given proper extended real function. This derivative is harmonized with the classical higher-order Fr\'echet directional derivative in the sense that both derivatives of the same order coincide if the last one exists.  We obtain necessary and sufficient conditions of order $n$ ($n$ is a positive integer) for a local minimum and isolated local minimum of order $n$ of the given function in terms of these derivatives and subdifferentials. We do not require any restrictions on the function in our results. A notion of a higher-order critical direction is introduced. It is applied in the characterizations of the isolated local minimum of order $n$. 
Higher-order invex functions are defined. They are the largest class such that our necessary conditions for local minima are sufficient for global one. We compare our results with some previous ones. 

As an application, we improve a result due to V. F. Demyanov, showing that the condition introduced by this author is a complete characterization of isolated local minimizers of order $n$.
\end{abstract}



\section{Introduction}
\label{s1}

One of the main tasks of non-differentiable optimization is to extend some optimality conditions to more general classes of non-differentiable functions.  There exist several types of minimizers. In this work, we will pay our attention to the local minimizers, global minimizers and to the isolated local ones in unconstrained problems. 

There are necessary and sufficient conditions in unconstrained optimization  in terms of various generalized derivatives. The most of them are of first- and second-order. The  higher-order conditions are rather limited. 
Such results were obtained in \cite{aub90,dem00,gin02,ban01,hof78,lin82,pal91,sad94,stu86,war94}. 
Even the conditions of first- and second-order  are satisfied for restricted classes of functions when are applied the known directional derivatives: locally Lipschitz, continuously differentiable, lower semicontinuous, the class C$^{1,1}$, and so on. 
The reader can find a lot of information concerning various types of higher-order derivatives, applied in mathematics, from the book \cite{muk12}.

The following properties are desirable when we compare the generalized derivatives to evaluate the optimality conditions:

1) The optimality conditions are necessary and sufficient.

2) The derivative extends the classical Fr\'echet derivative, that is the directional derivative coincides with the  Fr\'echet directional derivative in the case when the last one exists.

3) The derivative is defined for an arbitrary function, not restricted to some class of functions.

4) Higher-order derivatives can be defined and they satisfy Conditions 1, 2, and 3.

5) Useful calculus rules are derived in terms of these derivatives.

In Refs. \cite{hof78,lin82,sad94} were obtained only necessary optimality conditions, and there are no sufficient ones in these works.
The sufficient conditions in \cite{aub90} concern the so called weak local minimizers which are not local minimizers. The lower Dini directional derivatives are often applied in optimization, but even the second-order sufficient conditions in terms of Dini derivatives need more assumptions like the function should be $l$-stable  \cite{pas08}. 

The derivatives in \cite{aub90,dem00,gin02,stu86,war94} do not extend  the classical Fr\'echet derivatives. For example, even for every twice Fr\'echet differentiable function the second-order directional derivative $\gin 2 f(x;u)$ of the function $f$ at the point $x$ in direction $u$ in \cite{gin02} does not coincide with the second-order Fr\'echet directional derivative. It is shown in \cite{NA} that in this case $\gin 2 f(x;u)=-\infty$ for every direction $u$, if $\nabla f(x)\ne 0$.   

The higher-order derivatives in \cite{ban01} are consistent with the classical Fr\'echet directional derivatives, but the sufficient conditions are very restrictive. Even the second-order conditions can be applied only for C$^{1,1}$ functions.

In this paper, we introduce a new generalized directional derivative of order $n$ ($n$ is a positive integer) which satisfies Properties 1, 2, 3, and 4. In our knowledge, there is no a derivative which fulfills all these properties. The second-order derivative, when the function is lower semicontinuous, coincides with the second-order epi-derivative due to Rockafellar \cite{roc88}. The necessary conditions for optimality and the sufficient ones hold for an arbitrary function not necessarily restricted to some class. We obtain necessary conditions for a local minimum, sufficient ones for a strict local minimum, and complete characterizations of isolated local minimizers of order $n$ ($n$ is a positive integer) in terms of this derivative.  We derive our criteria for arbitrary proper extended real functions. The convergence in the definition of the derivatives  is of Hadamard type. We introduce a subdifferential of order $n$ and apply it in the optimality criteria. We should mention that optimality conditions for isolated local minima of order $n$ were obtained by Studniarski \cite{stu86}, but in contrast of our derivative, the derivatives of lower order do not appear in the formulation of the derivative of order $n$. We  compare our conditions for isolated local minimum of order $n$ with Theorem 2.1 in \cite{stu86}. We introduce a notion of a critical direction of order $m$ ($m$ is positive integer) and apply it in the characterizations of the isolated local minima of order $n$. Therefore our Theorem \ref{th2} is quite different from \cite[Theorem 2.1]{stu86}. We continue the investigation with Theorem \ref{th-Z}.
We additionally prove conditions of order $n$, which are both necessary and sufficient for a given point to be a global minimizer. They concern a new class of invex functions of order $n$. It is an extension of the respective notion in \cite{applanal-1,applanal-2} 
At last, we compare our optimality conditions with some known results (see Propositions \ref{pr5} and \ref{pr6}). 

We also improve the main result from Ref. \cite{dem00} due to V. F. Demyanov. Demyanov introduced the derivative $f^\downarrow_n(x)$ of order $n$ of the function $f$ at the point $x$. He proved that the condition $f^\downarrow_n(x)>0$ implies that $x$ is a strict local minimizer. Such condition is not sensible to locate strict minimizers, because the first-order derivative is enough to find them. We show that this condition is a complete characterization of isolated local minimizers of order $n$. Now, we can differentiate various types of strict minimizers.

In the work by Ivanov \cite{NA} are obtained second-order optimality conditions for scalar and vector problems using the particular case when the order of the derivative is second. The epi-derivative due to Rockafellar \cite{roc88} is defined for lower semicontinuous function. The derivative in \cite{NA} extends the epi-derivative due to Rockafellar to an arbitrary function, which does not belong to some restricted class of functions. It is proved that the optimality conditions for unconstrained problems in the papers \cite{bp04,pas08,bt85,cha94,cha87,cc90,ggr06,husn84,hua94,jl98,roc88,roc89,yan96,yan99,yj92} are consequence of the optimality conditions in terms of the derivative in \cite{NA}.

\section{Higher-order directional derivatives and subdifferentials of Ha\-da\-mard type}
\label{s2}

In this paper, we suppose that $\E$ is a real finite-dimensional Euclidean space. Denote by $\R$ the set of reals and $\RR=\R\cup\{-\infty\}\cup\{+\infty\}$.
Let $X$ and $Y$ be two linear spaces and $L(X,Y)$ be the space of all continuous linear operators from $X$ to $Y$. Then denote by $L^1(\E)$ the space $L(\E,\R)$, by $L^2(\E)$ the space $L(\E,L^1(\E))$ and so on. If $n$ is an arbitrary positive integer such that $n>1$, let $L^n(\E)$ be the linear space $L(\E,L^{n-1}(\E))$.
 
\begin{definition}[\cite{atf87}]
Let $X$ and $Y$ be real normed spaces. Consider a map $F:X\to Y$. If there exists a neighborhood of the point $x\in X$ such that the map $F$ can be represented in the form
\[
F(x+u)=F(x)+\Lambda(u)+\alpha(u)\norm{u},
\]
where $\Lambda\in L(X,Y)$ and $\lim_{\norm{u}\to 0}\norm{\alpha(u)}=0$, then $\Lambda$ is called the Fr\'echet derivative of $F$ at $x$.
\end{definition}

\begin{definition}[\cite{atf87}]
Let $f:\E\to\R$ be a  Fr\'echet differentiable function, whose Fr\'echet derivative at the point $x$ is denoted by $\nabla f(x)$. Then the second-order Fr\'echet derivative is the map $\nabla^2 f$ such that $\nabla^2 f\in L^2(\E)$ and $\nabla^2 f(x)$ is the 
Fr\'echet derivative of $\nabla f(x)$. The higher-order Fr\'echet derivatives are defined recursively. Let $\nabla^{n-1}(x)$ be the derivative of order $(n-1)$ at $x$. Then $\nabla^n f\in L^n(\E)$ and $\nabla^n f(x):=\nabla (\nabla^{n-1}f(x))$.  
\end{definition}

\begin{Peano}[Taylor's formula with a reminder in the form of Peano, \cite{atf87}]
Let $f:\E\to\R$ be a multivariable function. Suppose that there exist the Fr\'echet derivatives 
\[
\nabla f(x),\; \nabla^2 f(x),\dots,\nabla^{n-1} f(x)
\]
for every $x$ from some neighborhood of the point $\hat x$ and there exists $\nabla^n f(\hat x)$. Then
\[
f(\hat x+u)=f(\hat x)+\sum_{k=1}^n\frac{1}{k!}\nabla^k f(\hat x)\underbrace{(u)(u)\dots (u)}_{k-\text{times}}+\alpha_n(\norm{u})
\norm{u}^n.
\]
where $\alpha_n$ is a such function that $\lim_{\norm{u}\to 0}\alpha_n(\norm{u})=0$.
\end{Peano}


Consider a proper extended real function $\f$, that is a function, which never takes the value $-\infty$ and at least one value is finite. The domain of a proper extended real function is the set:
\[
{\rm dom}\; f:=\{x\in\E\mid f(x)<+\infty\}.
\]

\begin{definition}\label{def1}
The lower Hadamard directional derivative of a function $\f$
at a point $x\in\dom f$ in direction $u\in\E$ is defined as follows:
\[
\ld 1 f(x;u)=\liminf_{t\downarrow 0,u^\pr\to u}\,
t^{-1}[f(x+t u^\pr)-f(x)].
\]
Here $t$ tends to 0 with positive values, and $u^\pr\to u$ implies that the norm $\norm{u^\pr-u}$ approaches $0$.
\end{definition}

\begin{definition}
Recall that the lower Hadamard subdifferential of a function $\f$ at some point
$x\in\dom f$ is defined  by the following relation:
\[
\lsubd 1 f(x)=\{ x^*\in L^1(\E)\mid x^*(u)\le\ld 1 f (x;u)\quad\textrm{for all directions}\quad u\in\E\}.
\]
\end{definition}

The following definition extends the epi-derivative due to Rockafellar \cite{roc88}, defined for lower semicontinuous functions, to an arbitrary function (see \cite{NA}):

\begin{definition}
Let $\f$ be an arbitrary proper extended real function.
Suppose that $x^*_1$ is a fixed element from the lower Hadamard subdifferential
$\lsubd 1 f (x)$ at the point $x\in\dom f$. Then the lower
second-order derivative of Hadamard type of $f$ at $x\in\dom f$ in direction $u\in\E$ is
defined as follows:
\[
\ld 2 f(x;x^*_1;u)=\liminf_{t\downarrow 0,u^\pr\to u}\,
2t^{-2}[f(x+t u^\pr)-f(x)-tx^*_1(u^\pr)].
\]
\end{definition}


\begin{definition}
Let $\f$ be an arbitrary proper extended real function. Suppose that $x\in\dom f$, $x^*_1\in\lsubd 1 f(x)$.
The lower second-order Hadamard subdifferential of the function $\f$ at the point
$x\in\dom f$ is defined  by the following relation:
\[
\lsubd 2 f(x;x^*_1)=\{ x^*\in L^2(\E)\mid x^*(u)(u)\le\ld 2 f (x;x^*_1;u)\quad\textrm{for all directions}\quad u\in\E\}.
\]
\end{definition}

We introduce the following definitions:

\begin{definition}\label{def3}
Let $\f$ be an arbitrary proper extended real function, and $n$ be any positive integer such that $n>1$. Suppose that the lower Hadamard subdifferential 
\[
\lsubd i f (x;x_1^*,x_2^*,\dots,x_{i-1}^*),\quad i=1,2,\dots, n-1
\]
of order $i$ at the point $x\in\dom f$ is nonempty and $x^*_i$ is a fixed point from it. 
Then the lower derivative of Hadamard type  of order $n$ of $f$ at $x\in\dom f$ in direction $u\in\E$ is defined as follows:
\[
\ld n f(x;x^*_1,x^*_2,\dots,x^*_{n-1};u)=\liminf_{t\downarrow 0,u^\pr\to u}\,\Delta_n,
\]
where
\[
\Delta_n=
n!\, t^{-n}\, [f(x+t u^\pr)-f(x)-\sum_{i=1}^{n-1}\frac{t^i}{i!}\, x^*_i\underbrace{(u^\pr)(u^\pr)\dots (u^\pr)}_{i-\text{times}}].
\]
This derivative is well defined as element of $\bar\R$, because only the term $f(x+t u^\pr)$ can be infinite in the expression for $\Delta_n$.
\end{definition}

\begin{definition}\label{def4}
Suppose that $\f$ is an arbitrary proper extended function, and $n$ is any positive integer.
Let $x^*_i$ be a fixed point from the lower Hadamard subdifferential
$\lsubd i f (x;x_1^*,x_2^*,\dots,x_{i-1}^*)$, $i=1,2,\dots, n-1$ of order $i$ at the point $x\in\dom f$. Then the lower
subdifferential of Hadamard type  of order $n$ of $f$ at $x\in\dom f$  is defined as follows:
\begin{gather*}
\lsubd n f(x;x^*_1,x^*_2,\dots,x^*_{n-1})=\{ x^*\in L^n(\E)\mid x^*\underbrace{(u)(u)\dots (u)}_{n-\text{times}} \\
\le\ld n f (x;x^*_1,x^*_2,\dots,x^*_{n-1};u),\;\forall u\in\E\}.
\end{gather*}
\end{definition}

The essence of the next result  is that the  derivatives, defined in Definition \ref{def3},
generalize the usual classical ones in contrast to the derivative in \cite{stu86} and a lot of other derivatives.

\begin{theorem}\label{pr3}
Let the function $\f$ have Fr\'echet derivatives
\[
\nabla f(y),\nabla^2 f(y),\dots,\nabla^{n-1} f(y)
\]
at each point $y\in\E$ from some neighborhood of the point $x\in\E$, and let
there exist the n-th order Fr\'echet derivative $\nabla^n f(x).$
Then the lower derivatives of order $m$ in every direction $u\in\E$ exist for every integer $m$ such that $1\le m\le n$ and we have the following relations:
\begin{gather}
\ld 1 f(x;u)=\nabla f(x)(u);\quad \lsubd 1 f(x)=\{\nabla f(x)\}; \notag \\
\ld m f(x;\nabla f(x),\nabla^2 f(x),\dots,\nabla^{m-1} f(x);u)=\nabla^m f(x)\underbrace{(u)\dots (u)}_{m-\text{times}},  m=2,3,\dots, n; \notag \\
\nabla^m f(x)\in\lsubd m f(x;\nabla f(x),\nabla^2 f(x),\dots,\nabla^{m-1} f(x)) ,\quad m=2,3,\dots, n. \label{28}
\end{gather}
\end{theorem}
\begin{proof}
The first-order relations are well known, because they concern the Hadamard directional derivative.

We prove by induction the relations of order $m>1$. Suppose that they are satisfied for every positive integer $k<m$. It follows from here that 
\[
\ld m f(x;\nabla f(x),\nabla^2 f(x),\dots,\nabla^{m-1} f(x);u)
\]
is well defined. By Taylor's expansion formula with a reminder in the form of Peano we have
\[
f(x+t u^\pr)=f(x)+\sum_{i=1}^m\,\frac{1}{i!}\, \nabla^i f(x)\underbrace{(tu^\pr)(tu^\pr)\dots (tu^\pr)}_{i-\text{times}}]+o(t^m),
\]
where $o(h)$ is a function such that $\lim_{h\to 0}\, o(h)/h=0$.
Then we conclude from Definition \ref{def3} that
\begin{gather*}
\ld m f(x;\nabla f(x),\nabla^2 f(x),\dots,\nabla^{m-1}f(x);u) \\
=\liminf_{t\downarrow 0,u^\pr\to u}\,[\nabla^m f(x)\underbrace{(u^\pr)(u^\pr)\dots (u^\pr)}_{m-\text{times}}+o(t^m)/t^m]=\nabla^m f(x)\underbrace{(u)(u)\dots (u)}_{m-\text{times}}.
\end{gather*}
By Definition \ref{def4} we obtain that Inclusions (\ref{28}) are satisfied.
\end{proof}

\section{Conditions for a local minimum}
\label{s3}
The following theorem contains necessary conditions for a local minimum.
\begin{theorem}\label{th1}
Let $\bar x\in\dom f $ be a local minimizer of the proper extended real function $\f$. Then
\begin{equation}\label{14}
0\in\lsubd 1 f(\bar x),\quad 0\in\lsubd n f(\bar x;\underbrace{0,0,\dots ,0)}_{(n-1)-times}\quad\text{for all}\quad n=2,3,4,\dots
\end{equation}
\end{theorem}
\begin{proof}
Since $\bar x$ is a local minimizer, then there exists a neighborhood $N\ni\bar x$ with $f(x)\ge f(\bar x)$ for all $x\in N$. Let $u\in\E$ be an arbitrary chosen direction. Then $f(\bar x+t u^\pr)\ge f(\bar x)$ for all sufficiently small positive numbers $t$ and for all directions $u^\pr$, which are  sufficiently close to $u$. It follows from Definition \ref{def1} that $\ld 1 f (\bar x;u)\ge 0$. Therefore $0\in\lsubd 1 f(\bar x)$, because $u\in\E$ is an arbitrary direction.

Let $n$ be an arbitrary positive integer and 
\[
0\in\lsubd i f(\bar x;\underbrace{0,0,\dots ,0)}_{(i-1)-times},\quad i=1,2,\dots,n-1
\]
Hence $\ld n f(\bar x;\underbrace{0,0,\dots ,0}_{(n-1)-times};u)$ has sense and
\[
\ld n f(\bar x;\underbrace{0,0,\dots ,0}_{(n-1)-times};u)=\liminf_{t\downarrow 0,u^\pr\to u,}\, n!\, t^{-n}
[f(\bar x+t u^\pr)-f(\bar x)]\ge 0,\quad\forall u\in\E.
\]
It follows from the definition of the lower subdifferential of order $n$ that
\[
0\in\lsubd n f(\bar x;\underbrace{0,0,\dots ,0)}_{(n-1)-times}. \qedhere
\]
\end{proof}

\begin{remark}
Condition (\ref{14}) is equivalent to the following one:
\begin{equation}\label{15}
\ld 1 f(\bar x;u)\ge 0,\quad\ld n f(\bar x;\underbrace{0,0,\dots ,0}_{(n-1)-times};u)\ge 0,
\end{equation}
for all $u\in\E$ for all positive integers  $n\ge 2$.

\end{remark}

We introduce the following definition:

\begin{definition}
We call every point $\bar x\in \dom f $ such that 
\[
\ld 1 f(\bar x;u)\ge 0,\quad\ld k f(\bar x;\underbrace{0,0,\dots ,0}_{(k-1)-times};u)\ge 0,\quad k=2,3,\dots,n\quad\textrm{ for all }u\in\E,
\]
a stationary of order $n$.
\end{definition}

The notion of a 1-stationary point coincides with the notion of a stationary point.

The following example shows that Condition (\ref{15}) is not sufficient for $\bar x$ to be a local minimizer:
\begin{example}\label{ex2}
Consider the function of one variable $f:\R\to\R$ defined by:
\[
f(x)=\left\{
\begin{array}{ll}
-\exp\,(-1/x^2), \quad & \textrm{if}\quad x\ne 0, \\
0, & \textrm{if}\quad x=0.
\end{array}\right.
\]
Let us take $\bar x=0$. Then we have
\[
\ld n f(\bar x;\underbrace{0,0,\dots ,0}_{(n-1)-times};u)=0\quad\textrm{ for all }u\in\R,\;\textrm{ for all positive integers } n,
\]
\[
\lsubd n f(\bar x;\underbrace{0,0,\dots ,0)}_{(n-1)-times}=\{0\}\;\textrm{if }n\textrm{ is odd,}\quad
\lsubd n f(\bar x;\underbrace{0,0,\dots ,0)}_{(n-1)-times}=(-\infty,0]\;\textrm{if }n\textrm{ is even.}
\]
Hence Condition {\rm (\ref{15})} is satisfied, but $\bar x$ is not a local minimizer. Really, it is a global maximizer.
\end{example}

A point $\bar x\in\dom f$ is said to be a strict local minimizer iff there exists a neighborhood $N\ni\bar x$ such that $f(x)>f(\bar x)$ for all $x\in N$ such that $x\ne\bar x$.  

The following conditions are sufficient  for strict local minimum:
\begin{theorem}\label{th4}
Let be given a proper extended real function $\f$ and a point $\bar x\in\dom f$. Suppose that for every direction $u\in\E$, $u\ne 0$ we have $\ld 1 f(\bar x;u)>0$,
or there exists a positive integer $n=n(u)$, $n\ge 2$, which depend on $u$, and such that the following conditions hold:
\begin{equation}\label{4}
0\in\lsubd 1 f(\bar x),\quad 0\in\lsubd i f(\bar x;\underbrace{0,0,\dots ,0)}_{(i-1)-times}\quad\text{for all}\quad i=1,2,\dots, n-1
\end{equation}
and
\begin{equation}\label{5}
\ld {n(u)} f(\bar x;\underbrace{0,0,\dots ,0}_{(n-1)-times};u)>0.
\end{equation}
Then $\bar x$ is a strict local minimizer.
\end{theorem}
\begin{proof}
Let $u\ne 0$ be an arbitrary direction. It follows from (\ref{4}) and (\ref{5}) that there exists $\alpha>0$ with
\[
\liminf_{t\downarrow 0,u^\pr\to u}\,n!\,t^{-n}[f(\bar x+t u^\pr)-f(\bar x)]>2\alpha>0.
\]
Therefore, there exist $\delta>0$ and $\eps>0$ such that
\begin{equation}\label{29}
f(\bar x+t u^\pr)\ge f(\bar x)+\alpha\, t^n/n!>f(\bar x)
\end{equation}
for every $t\in (0,\delta)$ and arbitrary $u^\pr$ with $\norm{u^\pr-u}<\eps$.

Without loss of generality we may suppose that $u$ belongs to the unit sphere $S:=\{u\in\E\mid\norm{u}=1\}$. Since $u$ is arbitrary chosen,  then we can cover $S$ by  neighborhoods $N(u;\eps):=\{u^\pr\in S\mid\norm{u^\pr-u}<\eps\}$ such that (\ref{29}) is satisfied. Taking into account that the unit sphere is compact, then we can choose a finite number of neighborhoods $N(u_1;\eps_1)$, $N(u_2,\eps_2)$,\dots $N(u_s;\eps_s)$ that cover $S$. Let the respective values of $\delta$ are $\delta_1$, $\delta_2$,\dots $\delta_s$ and $\bar \delta=\min\{\delta_i\mid 1\le i\le s\}$. Then we have
\[
f(\bar x+t u^\pr)>f(\bar x),\quad\forall u^\pr\in S,\;\forall t\in(0,\bar\delta). 
\]
Hence, $f(x)>f(\bar x)$ for all $x\in\E$ such that $\norm{x-\bar x}<\bar\delta$, which implies that $\bar x$ is a strict local minimizer.
\end{proof}

\section{Global optimality conditions with a higher-order invex function}
\label{s5}
Example \ref{ex2} shows that the necessary conditions for a local minimum are not sufficient for a global one.
Then the following question arises: Which is the largest class of functions such that the necessary optimality conditions  from Theorem \ref{th1} become sufficient for a global minimum. 
Recently, Ivanov \cite{applanal-1} introduced a new class of functions called higher-order invex ones in terms of the lower Dini  directional derivatives and applied them in sufficient optimality conditions for inequality-constrained nonlinear programming problems \cite{applanal-2}. They extend the so called invex ones. We generalize the notions invexity and higher-order invexity to arbitrary non-differentiable functions in terms of the lower Hadamard directional derivatives of order $n$. 

First, we recall the definition of an invex function \cite{han81} in terms of the lower Hadamard directional derivative.

\begin{definition}
 A proper extended real function $\f$ is called invex in terms of the lower Hadamard directional derivative iff there exists a map
$\eta_1: \E\times \E\to\E$  such that the following inequality holds for all $x\in\E $, $y\in\E$: 
\begin{equation}\label{13}
f(y)-f(x)\ge\ld 1 f(x;\eta_1(x,y)). 
\end{equation}
\end{definition}

We introduce the following definition:
\begin{definition}
We call a proper extended function $\f$  invex of order $n$ in terms of the lower Hadamard derivatives  iff for every $\bar x\in\dom f$, $x\in\E$ with
\[
0\in\lsubd 1 f(\bar x),\; 0\in\lsubd i f(\bar x;0,0,\dots, 0),\; i=2,3,\dots, n
\]
there are $\eta_1$, $\eta_2,\dots$, $\eta_n$, which depend on $\bar x$ and $x$ such that the  following inequality holds 
\begin{equation}\label{8}
f(x)-f(\bar x)\ge\ld 1 f(\bar x;\eta_1(\bar x,x))+\sum_{i=2}^n\ld i f(\bar x; 0,0,\dots,0 ;\eta_i(\bar x,x)).
\end{equation}

If there exist $\eta_1(\bar x,x)$, $\eta_2(\bar x,x)$, $\eta_3(\bar x,x),\dots$ such that (\ref{8}) is satisfied with $n=+\infty$, then we call $f$ invex in generalized sense (or invex of order $+\infty$).
\end{definition}

\begin{definition}
We call a point $x\in\dom f$, which satisfies the necessary conditions (\ref{15}) stationary point of order $n$. 
\end{definition}

\begin{theorem}\label{npi}
Let $n$ be a positive integer or $+\infty$ and  $\f$ a proper extended function.  Then $f$ is invex of order $n$ if and only if each stationary point $\bar x\in\dom f$ of order $n$ $(n<+\infty$ or $n=+\infty)$ is a global minimizer of $f$.
\end{theorem}
\begin{proof}
We prove the case $n<+\infty$. The other case is similar.
Suppose that $f$ is invex of order $n$. If the function has no stationary points, then obviously every stationary point is a global minimizer. Suppose that the function has at least one stationary point.
Suppose that $\bar x\in\dom f $ is a given stationary point of order $n$. We prove that it is a global minimizer of $f$.
Take an arbitrary point  $x$ from $\E$. It follows from invexity of order $n$ that there exist $\eta_i(\bar x,x)$, $i=1,2,\dots, n$ such that 
\begin{equation}\label{9}
f(x)-f(\bar x)\ge\ld 1 f(\bar x;\eta_1(\bar x,x))+\sum_{i=2}^n\ld i f(\bar x; 0,0,\dots,0;\eta_i(\bar x,x) ).
\end{equation}

Since $\bar x$ is a stationary point of order $n$, then 
\[
\ld 1 f(\bar x;u)\ge 0,\;\ld i f(\bar x; 0,0,\dots ,0;u)\ge 0\quad\text{for all}\quad i=2,3,\dots,n,\;\forall u\in\E.
\]
Hence
\[
\ld 1 f(\bar x;\eta_1(\bar x,x))\ge 0,\;\ld i f(\bar x; 0,0,\dots ,0;\eta_i(\bar x,x) )\ge 0\quad\text{for all}\quad i=2,3,\dots,n.
\]
It follows from (\ref{9}) that $f(x)\ge f(\bar x)$. Therefore $\bar x$ is a global minimizer.

Conversely, suppose that every stationary point of order $n$ is a global minimizer. We prove that $f$ is invex of order $n$. 
Assume the contrary.  Hence, there exists a pair $(\bar x,x)\in\dom f\times \E$ such that
\[
0\in\lsubd 1 f(\bar x),\; 0\in\lsubd i f(\bar x;0,0,\dots, 0),\; i=2,3,\dots, n,
\]
but 
\begin{equation}\label{10}
f(x)-f(\bar x)<\ld 1 f(\bar x;u_1)+\sum_{i=2}^n\ld i f(\bar x; 0,0,\dots,0;u_i ).
\end{equation}
for all $u_i\in\E$, $i=1,2,\dots,n$.

First, we prove that $f(x)<f(\bar x)$. Let us choose in (\ref{10}) $u_i=0$, $i=1,2,\dots,n$. We have 
\[
\ld 1 f(\bar x;0)\le\liminf_{t\downarrow 0}\,t^{-1}(f(\bar x+t.0)-f(\bar x))=0.
\]
Let $i$ be an arbitrary integer such that $1<i\le n$. Then
\begin{gather*}
\ld i f(\bar x;0,0,\dots,0;0)= \liminf_{t\downarrow 0,u^\pr\to 0}\,\frac{i!}{t^i}\,[f(\bar x+t.u^\pr)-f(\bar x)] \\
\le\liminf_{t\downarrow 0}\,\frac{i!}{t^i}\,[f(\bar x+t.0)-f(\bar x)]=0,\quad i=2,3,\dots, n.
\end{gather*}
 It follows from (\ref{10}) that $f(x)<f(\bar x)$.

Second, we prove that  
\begin{equation}\label{11}
\ld 1 f(\bar x;u)\ge 0,\quad \forall u\in\E.
\end{equation} 
Suppose the contrary that there exists at least one point $v\in\E$ with $\ld 1 f(x;v)<0$. The lower Hadamard directional derivative is positively homogeneous with respect to the direction, that is 
\[
\ld 1 f(\bar x;\tau u)=\tau\ld 1 f(\bar x;u),\quad\forall \bar x\in\dom f,\;\forall u\in\E,\;\forall \tau\in(0,+\infty).
\]
Then inequality (\ref{10}) is satisfied when $u_1=tv$, $t>0$, $u_i=0$, $i\ne 1$, that is 
\[
f(x)-f(\bar x)<t\ld 1 f(\bar x;v),\quad\forall t>0,
\]
which is impossible, because $f(x)-f(\bar x)$ is finite and $\ld 1 f(x;v)<0$.  
Therefore, $\ld 1 f(\bar x;u)\ge 0$ for all $u\in\E$.

Third, we prove that 
\begin{equation}\label{12}
\ld i f(\bar x;0,0,\dots,0;u)\ge 0,\quad i=2,3,\dots, n
\end{equation}
for all $u\in\E$.
Suppose the contrary that there exists $v\in\E$ with $\ld i f(\bar x;0,0,\dots,0;v)<0$. The lower Hadamard directional derivative of order $i$ is positively homogeneous of degree $i$ with respect to the direction, that is
\[
\ld i f(\bar x;0,0,\dots,0;tv)=t^i\ld i f(\bar x;0,0,\dots,0;v),\quad\forall t>0.
\]
Then it follows from (\ref{10}) with $u_i=tv$, $t>0$, $u_k=0$ when $k\ne i$ that
\[
f(x)-f(\bar x)<t^i\ld i f(\bar x;0,0,\dots,0;v ),\quad\forall t>0,
\]
which is impossible when $t$ is sufficiently large positive number.

The following is the last part of the proof.   It follows from (\ref{11}) and (\ref{12}) that $\bar x$ is a stationary point of order $n$. According to the hypothesis $\bar x$ is a global minimizer, which contradicts the inequality $f(x)<f(\bar x)$. 
\end{proof}

In the next claim we show that the class of invex functions of order $(n+1)$ contains all invex functions of order $n$ in terms of the lower Hadamard directional derivative.

\begin{proposition}
Let $\f$ be an invex function of order $n$. Then $f$ is invex of order $(n+1)$. Every invex function of order $n$ is invex of order $+\infty$.
\end{proposition}
\begin{proof}
It follows from Equation (\ref{8})  that $f$ is invex of order $(n+1)$ keeping the same maps $\eta_1$, $\eta_2$,..., $\eta_n$ and taking $\eta_{n+1}=0$, because $\ld {n+1} f(\bar x;0,0,\dots,0;0)\le 0$.
\end{proof}

The converse claim is not satisfied. There are a lot of second-order invex functions, which are not invex. The following example is extremely simple.
\begin{example}
Consider the function $f:\R^2\to\R$ defined by 
\[
f(x_1,x_2)=-x_1^2-x_2^2.
\]
We have $\ld 1 f(x;u)=-2x_1 u_1-2x_2 u_2$ where $u=(u_1,u_2)$ is a direction. Its only stationary point is $\bar x=(0,0)$. This point is not a global minimizer. Therefore, the function is not invex. We have $\lsubd 1 f(\bar x)\equiv \{(0,0)\}$ and $\ld 2 f(\bar x;0;u)=-2u_1^2-2u_2^2$. It follows from here that $f$ has no second-order stationary points. Hence, every second-order stationary point is a global minimizer, and the function is second-order invex.
\end{example}

\begin{example}\label{npc}
Consider the function $f_n:\R\to\R$, where $n\ge 2$ is a positive integer:
\[
f_n(x)=\left\{
\begin{array}{rr}
x^n\,, & x\ge 0\,, \\
(-1)^{n-1}x^n\,, & x<0\,.
\end{array}\right.
\]
If $n$ is an odd number, then $f_n=x^n$. For $n$ even $f_n$ is a function from the class C$^{n-1}$, but not from the class C$^n$.
It has no stationary points of order $n$. Therefore, every stationary point of order $n$ is a global minimizer. According to Theorem \ref{npi}, it is invex of order $n$. On the other hand, the point $x=0$ is stationary of order $(n-1)$. Taking into account that $x=0$ is not a global minimizer, we conclude from the same theorem that the function is not invex of order $(n-1)$.
\end{example}

\section{Characterizations of the isolated minimizers}
\label{s4}
The following definition was introduced by Studniarski \cite{stu86} as a generalization of the respective notion of order $1$ and $2$ in \cite{hes75,aus84}.

\begin{definition}
Let $n$ be a positive integer.
A point $\bar x\in\dom f$ is called an isolated local minimizer of
order $n$ for the function $\f$ iff there exist a neighborhood $N$
of $\bar x$ and a constant $C>0$ with
\begin{equation}\label{1}
f(x)\ge f(\bar x)+C\norm{x-\bar x}^n,\quad\forall  x\in N.
\end{equation}
\end{definition}

\begin{definition}[\cite{stu86}]
A direction $d\in\E$ is said to be critical at the point $x\in\E$ iff $\ld 1 f(x;d)\le 0$.
\end{definition}

We introduce the following notion. 

\begin{definition}
Let $x\in\dom f$ be a stationary point of order  $n$. We call a direction $d\in\E$ critical at the point $x\in\E$ of order $m$ ($m$ is positive integer), $m\le n$ iff 
\[
\ld 1 f(x;d)\le 0,\quad  \ld 2 f(x;0;d)\le 0,\;\dots\;, 
 \ld {m} f(\bar x;\underbrace{0,0,\dots ,0}_{(m-1)-times};d)\le 0.
\]
\end{definition}

The notion of a critical direction coincides with the notion of a critical direction of order $1$. If a direction $d$ is critical of order $m$, then it is critical of order $m-1$. The inverse claim is not satisfied. For example, the function $f_n$ from Example \ref{npc} has a critical direction $u=1$ of order $n-1$ at $x=0$, but this direction is not critical of order $n$.

\begin{theorem}\label{th2}
Let be given a proper extended real function $\f$ and $\bar x\in\dom f$. Then the following three conditions are equivalent:

a) $\bar x$ is an isolated local minimizer of order $n$, where $n$ is a positive integer such that $n\ge 2$;

b) (\ref{4}) is satisfied
and
\begin{equation}\label{35}
\ld {n} f(\bar x;\underbrace{0,0,\dots ,0}_{(n-1)-times};u)>0,\quad\forall\; u\in\E\setminus\{0\};
\end{equation}

c) Inequality (\ref{35}) is satisfied for every critical of order $n$ direction $u$ such that $u\ne 0$.
\end{theorem}
\begin{proof}
We prove the claim a) $\Rightarrow$ b).
Let $\bar x$ be an isolated local minimizer of order $n$. We prove that Conditions (\ref{4}) and (\ref{35}) hold.
Suppose that $u\in\E$ is arbitrary chosen. It follows from Inequality (\ref{1}) that there exist numbers $\delta>0$, $\varepsilon>0$ and $C>0$ with
\begin{equation}\label{3}
f(\bar x+tu^\pr)\ge f(\bar x)+Ct^n\norm{u^\pr}^n
\end{equation}
for all  $t\in (0,\delta)$ and every $u^\pr$ such that $\norm{u^\pr-u}<\varepsilon$. Therefore 
\begin{equation}\label{2}
\ld 1 f (\bar x;u)=\liminf_{t\downarrow 0,u^\pr\to u,}\,t^{-1}[f(\bar x+t u^\pr)-f(\bar x)]\ge\liminf_{t\downarrow 0,u^\pr\to u,}\, Ct^{n-1}\norm{u^\pr}^{n}=0.
\end{equation}
Therefore $0\in\lsubd 1 f(\bar x)$. 

Suppose that $m$ is an arbitrary positive integer such that $1\le m<n$ and we have
\[
0\in\lsubd i f(\bar x;\underbrace{0,0,\dots ,0)}_{(i-1)-times}\quad\text{for all}\quad i<m.
\]
It follows from (\ref{3}) that
\begin{gather*}
\ld m f(\bar x;\underbrace{0,0,\dots ,0}_{(m-1)-times};u) 
=\liminf_{t\downarrow 0,u^\pr\to u,}\,m!\, t^{-m}[f(\bar x+t u^\pr)-f(\bar x)] \\
\ge\liminf_{t\downarrow 0,u^\pr\to u,}\, Ct^{n-m}\norm{u^\pr}^{n}=0, \;\forall u\in\E.
\end{gather*}
and $0\in\lsubd m f(\bar x;\underbrace{0,0,\dots ,0)}_{(m-1)-times}$. Then it follows from (\ref{3}) that
\[
\ld n f(\bar x;\underbrace{0,0,\dots ,0}_{(n-1)-times};u)\ge\liminf_{t\downarrow 0,u^\pr\to u,}\, C\norm{u^\pr}^{n} >0,\quad\forall u\in\E\setminus\{0\}.
\] 

The claim b) $\Rightarrow$ c) is obvious.

At last, we prove the claim c) $\Rightarrow$ a).
Suppose that Conditions {\rm (\ref{35})}  is satisfied for every critical direction $u\ne 0$ of order $(n-1)$. We prove that $\bar x$ is an isolated local minimizer of order $n$. Assume the contrary that $\bar x$ is not an isolated minimizer of order $n$.
Therefore, for every sequence $\{\varepsilon_k\}_{k=1}^\infty$ of positive numbers
converging to zero, there exists a sequence $\{x_k\}$ with $x_k\in\dom f$ such that
\begin{equation}\label{6}
\norm{x_k-\bar x}\le \varepsilon_k,\quad
f(x_k)< f(\bar x)+\varepsilon_k\norm{x_k-\bar x}^n,
\end{equation}

It follows from (\ref{6}) that $x_k\to\bar x$. Denote $t_k=\norm{x_k-\bar x}$,
$d_k=(x_k-\bar x)/t_k$. Passing to a subsequence, we may suppose
that $d_k\to  d$ where $\norm{d}=1$. It follows from here that
\begin{gather*}
\ld 1 f (\bar x;d)\le\liminf_{k\to\infty}\, t^{-1}_k[f(\bar x+t_k d_k)-f(\bar x)] \\
=\liminf_{k\to\infty}\, t^{-1}_k[f(x_k)-f(\bar x)]
\le\liminf_{k\to\infty}\,\varepsilon_k t_k^{n-1}=0.
\end{gather*}
It follows from $0\in\subd f(\bar x)$ that $\ld 1 f (\bar x;d)=0$. 

Let $m$ be any integer with $1<m\le n$ such that $\ld i f(\bar x;\underbrace{0,0,\dots ,0}_{(i-1)-times};d)=0$ for $i<m$. Therefore
\begin{equation}\label{16}
\begin{array}{l}
\ld {m} f(\bar x;\underbrace{0,0,\dots ,0}_{(m-1)-times};d)\le\liminf_{n\to+\infty}\,m!\,t^{-m}[f(x_k)-f(\bar x)] \\
\le\liminf_{k\to\infty}\, m!\,\varepsilon_k\, t_k^{n-m}=0,
\end{array}
\end{equation}
because $n-m\ge 0$ and $\varepsilon_k\to 0$. Then it follows from (\ref{4}) that
\[
\ld m f(\bar x;\underbrace{0,0,\dots ,0}_{(m-1)-times};d)=0\quad {\rm if}\quad m<n.
\]
Therefore  the direction $d$ is critical of order $n$.
 We conclude from the case $m=n$ that Inequality (\ref{16}) contradicts Condition (\ref{35}).
\end{proof}

\begin{example}
Consider the function $f:\R^2\to\R$ defined by
\[
f(x_1,x_2)=\left\{
\begin{array}{ll}
\exp\,(-1/(x_1^2+x_2^2)), & \textrm{if}\quad (x_1,x_2)\ne (0,0), \\
0, & \textrm{if}\quad (x_1,x_2)=(0,0).
\end{array}\right.
\]
The point $\bar x=(0,0)$ is a strict global minimizer, but there is no a positive integer $n$ such that $\bar x$ is an isolated minimizer of order $n$. This fact can be established directly from the definition of an isolated minimizer of order $n$, but it also follows from Theorem \ref{th2}.
\end{example}

\begin{remark}
Studniarski introduced in Ref. \cite{stu86} the following directional derivative of order $n$ at the point $x$ in direction $u$:
\[
\underline d^n f(x;u)=\liminf_{t\downarrow 0,u^\pr\to u}\, t^{-n}\, [f(x+t u^\pr)-f(x)].
\]
He derived necessary and sufficient optimality conditions for isolated local minima of order $n$ for unconstrained problems in term of this derivative. Obviously, this derivative does not coincide with the Fr\'echet directional derivative of order $n$ in the case when the last one exists. On the other hand, we have 
\[
\ld n f(x;0,0,\dots,0;u)=n!\underline d^n f(x;u).
\]

Theorem 2.1 in \cite{stu86} and Theorem \ref{th2} are similar but different.
The essential  difference  is that in Condition c) Studniarski supposed that the inequality is satisfied for every critical direction $d$, $d\ne 0$. Since every critical direction of order $n-1$ is critical, then \cite[Theorem 2.1]{stu86} is a consequence of Theorem \ref{th2}.
\end{remark}

The following derivative of order $k$ ($k$ is positive integer) was introduced by Demyanov \cite{dem00}:
\[
f^{\downarrow}_k(x)=\liminf_{y\to x,y\ne x}\frac{f(y)-f(x)}{\norm{y-x}^k}.
\]

The following one is the main result in \cite{dem00}:

\begin{proposition}\label{pr1}
For a point $x\in\dom f$ to be a global or local minimizer of the function $f$ on a metric space  $X$ it is necessary that $f^{\downarrow}_k(x)\ge 0$ for every positive integer $k$.

If for some positive integer $k$ it turns out that $f^{\downarrow}_k(x)>0$, then $x$ is a strict local minimizer of $f$ on $X$.
\end{proposition}

The reader may compare this proposition with Lemma \ref{lema2} below. Is it sensible to locate strict minimizers using the derivatives of order higher than first, if we do not differentiate them? We prove that the condition $f^{\downarrow}_k(x)>0$ is a complete characterization of the isolated local minimizers of order $k$ applying Theorem \ref{th2}.

Really, the following relation holds between the derivatives of Studniarski and Demyanov:
\begin{proposition}\label{prdem}
Let $x\in\dom f$ and $n$ be any positive integer. Then
\[
f^{\downarrow}_n(x)=\min_{u\in S}\underline d^n f(x;u),\textrm{ where } S:=\{u\in\E\mid\norm{u}=1\}
\]
\end{proposition}
\begin{proof}
Take an arbitrary point $u\in S$. We have
\begin{gather*}
\liminf_{t\downarrow 0,u^\pr\to u}\frac{f(x+t u^\pr)-f(x)}{t^n}=\liminf_{t\downarrow 0,u^\pr\to u}\frac{f(x+t u^\pr)-f(x)}{(t\norm{u^\pr})^n}\\
\ge\liminf_{y\to x,y\ne x}\frac{f(y)-f(x)}{\norm{y-x}^n}=f^{\downarrow}_n(x).
\end{gather*}
Therefore $f^{\downarrow}_n(x)\le\underline d^n f(x;u)$ for every $u\in S$. It follows from here that
\[
f^{\downarrow}_n(x)\le\inf_{u\in S}\underline d^n f(x;u).
\]

We prove the inverse inequality. Let us take an arbitrary infinite sequence $\{\eps_i\}_{i=1}^{\infty}$ such that $\eps_i>0$. There exists a sequence $y_i$ such that $y_i\to x$, $y_i\ne x$ and 
\[
f(y_i)-f(x)]/\norm{y_i-x}^n<f^{\downarrow}_n(x)+\eps_i.
\]
We can choose $y_i$ such that 
\[
f^{\downarrow}_n(x)=\lim_{i\to\infty} [f(y_i)-f(x)]/\norm{y_i-x}^n.
\]
Denote $t_i=\norm{y_i-x}$ and $v_i=(y_i-x)/t_i$. Without loss of generality we could suppose that $v_i\to v\in S$, because $\norm{v_i}=1$.
It follows from here that

\[
\underline d^n f(x;v)\le\liminf_{i\to+\infty}\frac{f(x+t_i v_i)-f(x)}{t_i^n}=\liminf_{i\to+\infty}\frac{f(y_i)-f(x)}{\norm{y_i-x}^n}=f^{\downarrow}_n(x).
\]
Therefore, $\inf_{u\in S}\underline d^n f(x;u)=f^{\downarrow}_n(x)=\underline d^n f(x;v)$. 
\end{proof}

\begin{proposition}\label{pr4}
Let be given a proper extended real function $\f$ and $\bar x\in\dom f$. Then $\bar x$ is an isolated local minimizer of order $n$, where $n$ is a positive integer, if and only if $f^{\downarrow}_n(\bar x)>0$.
\end{proposition} 
\begin{proof}
Let $\bar x$ be an isolated local minimizer of order $n$. We prove that $f^{\downarrow}_n(\bar x)>0$. By Proposition \ref{prdem} there exists a $v\in S$ such that
\[
f^{\downarrow}_n(\bar x)=\underline d^n f(\bar x;v)=\liminf_{t\downarrow 0,v^\pr\to v}\, t^{-n}\, [f(\bar x+t v^\pr)-f(\bar x)].
\] 
According to Theorem \ref{th2} we have 
\[
f^{\downarrow}_n(\bar x)=\underline d^n f(x;v)=\ld {n} f(\bar x;\underbrace{0,0,\dots ,0}_{(n-1)-times};v)>0.
\]

Suppose that $f^{\downarrow}_n(\bar x)>0$. We prove that $\bar x$ is an isolated local minimizer of order $n$.
By Proposition \ref{prdem} we obtain $\min_{u\in S}\underline d^n f(x;u)=f^{\downarrow}_n(x)>0$.
By the relation $\ld n f(\bar x;\underbrace{0,0,\dots,0}_{(n-1)-times};u)=n!\underline d^n f(\bar x;u)$ we conclude from Theorem \ref{th2} that $x$ is an isolated local minimizer of order $n$.
\end{proof}

Now we investigate more detailed the isolated minimizers.
\begin{lemma}\label{lema2}
Let $\bar x\in\dom f $ be an isolated  local minimizer of order $n$ of the proper extended real function $\f$. Then $\bar x$ is an isolated local minimizer of order $n+1$.
\end{lemma}
\begin{proof}
Let $\bar x$ be an isolated local minimizer of order $n$. It follows from the definition that there exist a neighborhood $N$
of $\bar x$ and a constant $C>0$ which satisfy inequality (\ref{1}).
Let $\eps=\sup\{\norm{x-\bar x}\mid x\in N\}$. Obviously $\eps>0$. Then
\[
[f(x)-f(\bar x)]/\norm{x-\bar x}^{n+1}>C/\eps,\quad\forall x\in N
\]

\noindent
which implies that $\bar x$ is an isolated local minimizer of order $n+1$ with a constant $C/\eps$ and the same neighborhood $N$.
\end{proof}

The following question arises from Lemma \ref{lema2}: How can we characterize the isolated local minimizers of order $n$ and such that $n$ is the least possible such integer?

\begin{theorem}\label{th-Z}
Let be given a proper extended real function $\f$ and $\bar x\in\dom f $. Then the following claims are equivalent:

a) $\bar x$ is an isolated  local minimizer of order $n$  and for every positive integer $k<n$ $\bar x$ is not an isolated local minimizer of order $k$;

b) Conditions (\ref{4}) and (\ref{35}) are satisfied
and for every $k<n$ there exists a direction $d_k\ne 0$ such that $\ld k f(\bar x;\underbrace{0,0,\dots ,0}_{(k-1)-times};d_k)=0$;

c) $f^{\downarrow}_k(\bar x)=0$ for every $k<n$ and $f^{\downarrow}_n(\bar x)>0$.
\end{theorem}
\begin{proof}
We prove the equivalence between the claims a) and b).
Let $\bar x$ be an isolated minimizer of order $n$ and it is not an isolated minimizer of order $k$ for every $k<n$. Therefore, by Theorem \ref{th2} Conditions (\ref{4}) and (\ref{35}) hold. Since it is not an isolated minimizer of order $k$ for every $k<n$, then according to Theorem \ref{th2} for every $k<n$ there exists $d_k$ such that $\ld k f(\bar x;\underbrace{0,0,\dots ,0}_{(k-1)-times};d_k)=0$. 

The converse claim is also an easy  consequence of Theorem \ref{th2}. Conditions (\ref{4}) and (\ref{35}) imply that $\bar x$ is an isolated minimizer of order $n$. The condition that for every $k<n$ there exists $d_k\ne 0$ such that $\ld k f(\bar x;\underbrace{0,0,\dots ,0}_{(k-1)-times};d_k)=0$ implies that $\bar x$ is not an isolated minimizer of order $k$.

We prove the equivalence between the claims a) and c).
Let $\bar x$ be an isolated minimizer of order $n$ and it is not an isolated minimizer of order $k$ for every $k<n$. Hence $\bar x$ is a local minimizer. Then we obtain from the necessary optimality conditions (see proposition \ref{pr1}) that $f^{\downarrow}_k(\bar x)\ge 0$ for every $k=1$, $2,$ $\dots$, $n$. Taking into account that $\bar x$ is an isolated minimizer of order $n$, we conclude from Proposition \ref{pr4} that
$f^{\downarrow}_n(\bar x)>0$. Assume that $f^{\downarrow}_k(\bar x)>0$ for some $k<n$. Then by Proposition \ref{pr4} $\bar x$ is an isolated minimizer of order $k$ which is a contradiction. Therefore  $f^{\downarrow}_k(\bar x)=0$ for every $k<n$.

We prove the converse claim. Let $f^{\downarrow}_k(\bar x)=0$ for every $k<n$ and $f^{\downarrow}_n(\bar x)>0$. Therefore $\bar x$ is an isolated minimizer of order $n$. Suppose that there is an integer $k<n$ such that $\bar x$ is an isolated minimizer of order $k$. Then by Proposition \ref{pr4} $f^{\downarrow}_k(\bar x)>0$ which is a contradiction.
\end{proof}

\section{Comparison with some previous results}
\label{s7}

The lower Dini directional derivative (in short, Dini derivative) at the point $x\in\dom f$ in direction $u$ is defined by the following equality:
\[
\dini 1 f(x,u)=\liminf_{t\downarrow 0}\,t^{-1}[f(x+t u)-f(x)].
\]
The lower Dini directional derivative of order $n$ at the point $x\in\dom f$ in direction $u$ is defined as follows:
\[
\dini n f(x,u)=\liminf_{t\downarrow 0}\,n! t^{-n}[f(x+t u)-f(x)-\sum_{k=1}^{n-1}\dini k f(x,u)].
\]
It exists if all derivatives of lower order are finite.

\begin{proposition}
Let $n>1$ be a positive integer, $f$ be a proper extended real scalar function such the  Dini derivative of order $n$ and lower exist. Suppose that $x\in\dom f$ is a local minimizer. Then the following necessary conditions are satisfied:

\medskip
$\dini 1 f(x;u)\ge 0,\quad\forall u\in\E$\hfill {\rm (D$_1$)}

$\dini i f(x;u)=0,\; i<k\quad\Rightarrow\quad\dini k f(x;u)\ge 0$\hfill {\rm (D$_k$)}

\medskip\noindent
for all $k=2,3,\dots, n$.
\end{proposition}

We compare these conditions with our necessary ones

\medskip
$\ld 1 f(x;u)\ge 0,$\hfill {\rm (N$_1$)}

$\ld k f(x;\underbrace{0,0,\dots ,0}_{(k-1)-times};u)\ge 0$,\hfill {\rm (N$_k$)}

\medskip
\noindent
for all $k=2,3,\dots, n$.

\begin{proposition}\label{pr5}
Let $f$ be a proper extended real scalar function such that all Dini derivatives of order $n$  and lower exist. Suppose that at least one of the conditions {\rm (D$_k$)}, $k=1,2,\dots, n$ is not satisfied at the point $x\in\dom f$. Then all derivatives in the conditions {\rm (N$_k$)}, $k=1,2,\dots, n$ exist, but at least one of the conditions {\rm (N$_k$)}, $k=1,2,\dots, n$ fails.
\end{proposition}
\begin{proof}
Suppose that (D$_1$) fails. Therefore there exists a direction $u\in\E$ such that $\dini 1 f(x,u)<0$. It follows from here that
\[
\ld 1 f(x;u)\le\dini 1 f(x;u)<0
\]
and (N$_1$) is not satisfied.

Let us suppose that all conditions {\rm (D$_i$)}, $i<k$ are satisfied. Here $k$ is an arbitrary integer with $2\le k\le n$. We may suppose without loss of generality that {\rm (D$_k$)} fails. Therefore there exists $u\in\E$ such that
\[
\dini i f(x;u)=0,\; i<k\quad\textrm{and}\quad \dini k f(x;u)<0.
\]
We have  $\ld 1 f(x;u)\le\dini 1 f(x;u)=0$. If there is a $v\in\E$ with $\ld 1 f(x;v)<0$, then (N$_1$)  fails. Otherwise (N$_1$) holds,
$0\in\lsubd 1 f(x)$ and $\ld 2 f(x;0;v)$ is well defined for every $v\in\E$. We have
\[
\ld 2 f(x;0;u)\le\liminf_{t\downarrow 0}\, 2t^{-2}[f(x+tu)-f(x)]=\dini 2 f(x;u)=0.
\]
If there is $v\in\E$ such that $\ld 2 f(x;0;v)<0$, then (N$_2$) fails. Otherwise (N$_2$) holds, $0\in\lsubd 2 f(x;0)$ and $\ld 3 f(x;0,0;v)$ exists for every $v\in\E$. We continue in this way. At last, if there exists $v\in\E$ with $\ld {k-1} f(x;\underbrace{0,0,\dots ,0}_{(k-2)-times};u)<0$, then (N$_{k-1}$) fails. Otherwise $0\in\lsubd {k-1} f(x;\underbrace{0,0,\dots ,0}_{(k-2)})$ and $\ld k f(x;\underbrace{0,0,\dots ,0}_{(k-1)};v)$ is defined for every $v\in\E$. We have
\[
\ld k f(x;\underbrace{0,0,\dots ,0}_{(k-1)};u)\le\liminf_{t\downarrow 0}\, k!t^{-k}[f(x+tu)-f(x)]=\dini k f(x;u)<0,
\]
which implies that (N$_k$) is not satisfied. In the cases $k=2$ and $k=3$ the proof will be simpler.
\end{proof}
This result shows that our optimality conditions are preferable than the conditions in term of lower Dini derivatives, because if some point can be rejected as a potential candidate for a local minimum by Dini derivatives, then it  can be rejected by our derivatives. 

\begin{example}
Consider the function of two variables $f:\R^2\to\R$ defined by:
\[
f(x)=\left\{
\begin{array}{ll}
-x_2^n, \quad & \textrm{if}\quad x_2=x_1^2, \\
0, & \textrm{if}\quad x_2\ne x_1^2.
\end{array}\right.
\]
The point $\bar x=(0,0)$ is not a local minimizer. Let us calculate the Dini derivatives. We have
\[
\dini k f(\bar x,u)=0\quad\textrm{for all}\quad k=1,2,\dots,n,\quad\forall u\in\R^2.
\]
Therefore the Dini derivatives cannot reject $\bar x$ as possible minimizer. On the other hand
\[
\ld k f(x;\underbrace{0,0,\dots ,0}_{(k-1)};u)=0,\;\forall\; k<n,\;\forall u\in\R^2,\quad\ld n f(x;\underbrace{0,0,\dots ,0}_{(n-1)};u)=-n!u_2^n,
\]
where $u=(u_1,u_2)$. It is easy to see that our derivatives reject $\bar x$ as candidate for minimizer, because $\ld n f(x;\underbrace{0,0,\dots ,0}_{(n-1)};u)<0$ when $u_2>0$.
\end{example}

Ginchev \cite{gin02} has introduced the following directional derivatives of Hadamard type. Let be given an proper extended real function $\f$. The derivatives begin with the derivative of  order $0$:
\[
\gin 0 f(x;u):=\liminf_{t\downarrow 0,u^\pr\to u}\, f(x+t u^\pr).
\]
Let $n$ be a positive integer. Then the derivative of order $n$ at the point $x\in \dom f$ in direction $u\in\E$ is defined as follows: 
\[
\gin n f(x;u):=\liminf_{t\downarrow 0,u^\pr\to u}\,\frac{n!}{t^n}\,[f(x+t u^\pr)-\sum_{i=0}^{n-1}\frac{t^i}{i!}\gin i f(x;u)].
\]
In \cite[Theorem 1]{gin02} the author derived necessary optimality conditions of order $n$ for a local minimum under the assumption that the required derivatives exist. In \cite[Theorems 9 and 10]{gin02} are derived necessary and sufficient conditions for isolated local minimum of order $n$ ($n$ is positive integer). For every direction $u$, he consider the following conditions:

\medskip
$\gin 0 f(x;u)>f(x)$\hfill {\rm (G$_0$)}

$\gin 0 f(x;u)=f(x)$, $\gin 1 f(x,u)>0$\hfill {\rm (G$_1$)}

$\gin 0 f(x;u)=f(x)$, $\gin i f(x,u)=0$, $i<n$, $\gin n f(x,u)>0$\hfill {\rm (G$_n$)}

\medskip
We compare these conditions with our sufficient ones for isolated local minimum of order $n$:

\medskip
$0\in\lsubd 1 f(x)$, \hfill {\rm (S$_1$)}

$0\in\lsubd k f(x;\underbrace{0,0,\dots ,0)}_{(k-1)-times}$ for all $k=2,3,\dots,n-1$ \hfill {\rm (S$_k$)}

$\ld {n} f(x;\underbrace{0,0,\dots ,0}_{(n-1)-times};u)>0$, \hfill {\rm (S$_n$)}

\medskip
The following lemma is well-known.
\begin{lemma}\label{lema1}
Let $a_n$ and $b_n$ be two infinite sequences such that $a_n\ge 0$, $b_n\ge 0$. If there exists the limit $\lim_{n\to+\infty} a_n$, then
\[
\liminf_{n\to +\infty}(a_n b_n)=\lim_{n\to +\infty} a_n\; \liminf_{n\to +\infty} b_n.
\]
\end{lemma}

\begin{proposition}\label{pr6}
Let $f$ be a proper extended real scalar function which is lower semicontinuous at the point $x\in\dom f$. Let  the required Ginchev's derivatives in Conditions {\rm (G$_k$)}, $k=0,1,\dots,n$ exist. Suppose that for every direction $u\in\E$, $u\ne 0$ is satisfied  at least one of conditions {\rm (G$_k$)}, $k=0,1,\dots,n$ where $n$ is a fixed positive integer. Suppose additionally that $\gin i f(x;0)=0$ for all $i=1,2,\dots,n$. Then all required derivatives in conditions {\rm (S$_k$)} exist and all the conditions {\rm (S$_k$)}, $k=1,2,\dots, n$ are satisfied for every $u\in\E$, $u\ne 0$. 
\end{proposition}
\begin{proof}
The case $n=1$ is trivial. 

Let $n>1$. Suppose that for every direction $u\ne 0$ at least one of the conditions {\rm (G$_k$)}, $k=0,1,\dots,n$ holds. By lower semicontinuity of $f$ we conclude that $\gin 0 f(x;u)=f(x)$. Therefore {\rm (G$_0$)} cannot be satisfied. We prove that all the conditions {\rm (S$_k$)}, $k=1,2,\dots, n$ are satisfied. First, we prove $(S_1)$. We conclude from {\rm (G$_k$)}, $k=1,2,\dots,n$ and by $\gin 1 f(x;0)=0$ that $\gin 1 f(x,u)\ge 0$ for every direction $u\in\E$, which implies that {\rm (S$_1$)} holds. 

Suppose that $k$ is any positive integer such that $k\le n$ and the conditions $(S_1)$, $(S_2)$, $\dots$ , $(S_{k-1})$ are proved. We prove $(S_k)$. It follows from $(S_{k-1})$ that 
the derivative $\ld  {k} f(x;\underbrace{0,0,\dots ,0}_{(k-1)-times};u)$ is well defined for every $u\in\E$. Let $u\in\E$ be an arbitrary direction such that $u\ne 0$. By conditions $(G_k)$ there exists a positive integer $m$ such that
\[
\gin i f(x;u)=0\quad \textrm{for all}\quad i<m,\quad \gin m f(x;u)>0.
\]

Consider several cases:

1) Let $m<k$. Therefore $\ld m f(x;\underbrace{0,0,\dots ,0}_{(m-1)-times};u)=\gin m f(x;u)>0$. Then, by Lemma \ref{lema1}, $\ld  {k} f(x;\underbrace{0,0,\dots ,0}_{(k-1)-times};u)>0$.

2) Let $m=k$. Therefore $\ld k f(x;\underbrace{0,0,\dots ,0}_{(k-1)-times};u)=\gin k f(x;u)>0$. 

3) Let $m>k$. Therefore $\ld k f(x;\underbrace{0,0,\dots ,0}_{(k-1)-times};u)=\gin k f(x;u)=0$.

In all cases $(S_k)$ is satisfied, because $\ld k f(x;\underbrace{0,0,\dots ,0}_{(k-1)-times};0)=\gin k f(x;0)=0$.

In the  case  when $k=n$ the third case $m>k$ is impossible. Therefore, $(S_n)$ holds.
\end{proof}

\begin{remark}
We suppose in the claim that $f$ is lower semicontinuous at $x$, but every function is lower semicontinuous at the points which are local minimizers.

The assumption $\gin i f(x;0)=0$ for all $i=1,2,\dots,n$ is very natural. Indeed, by \cite[Theorem 10]{gin02}, the point $x$, which satisfies the hypothesis of the theorem, is an isolated local minimizer. On the other hand, by the necessary conditions \cite[Theorem 1]{gin02} it has to be satisfied the following inequality: $\gin i f(x;u)\ge 0$ for all $i=1,2,\dots,n$ for every $u\in\E$. In particular, $\gin i f(x;0)\ge 0$ for all $i=1,2,\dots,n$. Really, $\gin i f(x;0)$ cannot be strictly positive. Suppose that for some index $p$ we have 
\[
\gin 0 f(x;0)=f(x),\quad  \gin i f(x;0)=0,\; 0<i<p.
\]
Then
\[
\gin p f(x;0)=\liminf_{t\downarrow 0,u^\pr\to 0}\, p![f(x+t u^\pr)-f(x)]/(t^p)\le\liminf_{t\downarrow 0}\, p![f(x+t.0)-f(x)]/(t^p)=0.
\]
Therefore $\gin p f(x;0)=0$.
\end{remark}

\begin{corollary}
Let $f$ be a proper extended real scalar function which is lower semicontinuous at the point $x\in\dom f$ and the required Ginchev's derivatives exist and for every direction $u\in\E$ is satisfied  at least one of conditions {\rm (G$_k$)}, $k=0,1,\dots,n$ if and only if for every direction $u$ are satisfied the Conditions {\rm (S$_k$)}, $k=1,2,\dots, n$. 
\end{corollary}
\begin{proof}
The claim is a consequence of the proposition, Theorem \ref{th2}, and Theorem 9 in Ref. \cite{gin02}.
\end{proof}

\end{document}